\documentclass[11pt, reqno]{amsart}
\usepackage[dvipsnames]{xcolor}
\usepackage{mathtools}
\mathtoolsset{showonlyrefs}
\usepackage{etoolbox}
\usepackage{mathrsfs}

\usepackage{blindtext}
\usepackage{latexsym}
\usepackage[pagewise]{lineno}
\usepackage[utf8]{inputenc}
\usepackage{exscale}
\usepackage{amsmath}
\usepackage{amssymb}
\usepackage{amsfonts}
\usepackage{mathrsfs}
\usepackage{amsbsy}
\usepackage{relsize}
\usepackage{comment}
\PassOptionsToPackage{reqno}{amsmath}

\usepackage[inline]{enumitem}  
\usepackage[margin=1in]{geometry}

\newcommand{\ta}{\ensuremath{\mathtt{a}}}

\definecolor{mypink1}{rgb}{0.858, 0.188, 0.478}
\definecolor{mypink2}{RGB}{219, 48, 122}
\definecolor{mypink3}{cmyk}{0, 0.7808, 0.4429, 0.1412}
\definecolor{mygray}{gray}{0.6}
\definecolor{venetianred}{rgb}{0.78, 0.03, 0.08}
\definecolor{sapphire}{rgb}{0.03, 0.15, 0.4}
\definecolor{utahcrimson}{rgb}{0.83, 0.0, 0.25}
\definecolor{trueblue}{rgb}{0.0, 0.45, 0.81}
\definecolor{carminered}{rgb}{1.0, 0.0, 0.22}
\definecolor{cobalt}{rgb}{0.0, 0.28, 0.67}
\definecolor{cornflowerblue}{rgb}{0.39, 0.58, 0.93}

\usepackage[colorlinks, linkcolor=utahcrimson, citecolor=blue, urlcolor=blue, hypertexnames=false]{hyperref}

\newtheorem{thm}{Theorem}[section]
\newtheorem{lem}{Lemma}[section]

\newtheorem{defi}{Definition}[section]
\newtheorem{rem}{Remark}[section]

\newcommand{\R}{\mathbb R}

\newcommand{\bw}{\mathbf{w}}

\newcommand{\bs}{{\mathbf S}}
\newcommand{\mb}{{\mathscr B}}

\usepackage{lipsum} 
\usepackage{tikz,xcolor}

\definecolor{lime}{HTML}{A6CE39}
\DeclareRobustCommand{\orcidicon}{
	\begin{tikzpicture}
	\draw[lime, fill=lime] (0,0) 
	circle [radius=0.16] 
	node[white] {{\fontfamily{qag}\selectfont \tiny ID}};
	\draw[white, fill=white] (-0.0625,0.095) 
	circle [radius=0.007];
	\end{tikzpicture}
	\hspace{-2mm}
}
\foreach \x in {A, B}{\expandafter\xdef\csname orcid\x\endcsname{\noexpand\href{https://orcid.org/\csname orcidauthor\x\endcsname}
			{\noexpand\orcidicon}}
}

\numberwithin{equation}{section}

\title[Nonlocal semilinear parabolic equations]{On the Nonexistence of Global Solutions for Nonlocal  Parabolic Equations with Forcing Terms}

\author[ R. Ben Belgacem and M. Majdoub]{Rihab Ben Belgacem and Mohamed Majdoub\orcidB{}}

\address[R. Ben Belgacem]{Universit\'e de Tunis El Manar, Facult\'e des Sciences de Tunis, D\'epartement de math\'ematiques, Laboratoire \'equations aux d\'eriv\'ees partielles (LR03ES04), 2092 Tunis, Tunisie.}
\email{\sl {\textcolor{blue}{rihabbenbelgacem662@gmail.com}}}
\address[M. Majdoub]{Department of Mathematics, College of Science, Imam Abdulrahman Bin Faisal University, P. O. Box 1982, Dammam, Saudi Arabia.}
\address[M. Majdoub]{Basic and Applied Scientific Research Center, Imam Abdulrahman Bin Faisal University, P.O. Box 1982, 31441, Dammam, Saudi Arabia.}
\email{\sl {\textcolor{blue}{mmajdoub@iau.edu.sa}}}
\email{\sl {\textcolor{blue}{med.majdoub@gmail.com}}}

\begin{document}
\allowdisplaybreaks
\begin{abstract}
The purpose of this work is to analyze the well-posedness and blow-up behavior of solutions to the nonlocal semilinear parabolic equation with a forcing term:
\[
\partial_t u - \Delta u = \|u(t)\|_{q}^\alpha |u|^p + t^{\varrho} \mathbf{w}(x) \quad \text{in} \quad \mathbb{R}^N \times (0, \infty),
\]

where \(N \geq 1\), \(p, q \geq 1\), \(\alpha \geq 0\), \(\varrho > -1\), and \(\mathbf{w}(x)\) is a suitably given continuous function. 

The novelty of this work, compared to previous studies, lies in considering a nonlocal nonlinearity \(\|u(t)\|_{q}^\alpha |u|^p\) and a forcing term \(t^{\varrho} \mathbf{w}(x)\) that depend on both time and space variables. This combination introduces new challenges in understanding the interplay between the nonlocal structure of the equation and the spatio-temporal forcing term.

Under appropriate assumptions, we establish the global existence of solutions for small initial data in Lebesgue spaces when the exponent \(p\) exceeds a critical value. In contrast, we show that the global existence cannot hold for \(p\) below this critical value, provided the additional condition \(\int_{\mathbb{R}^N} \mathbf{w}(x) \, dx > 0\) is satisfied. The main challenge in this analysis lies in managing the complex interaction between the nonlocal nonlinearity and the forcing term, which significantly influences the behavior of solutions.\end{abstract}
\date{\today}

	\maketitle

\tableofcontents


\newpage
\section{Introduction and main results}
\label{S1}

The study of nonlinear parabolic equations has been a central topic in mathematical analysis due to their wide-ranging applications in real-word problems. Among these, the nonlinear heat equation with nonlocal source and/or forcing terms has garnered significant attention because of its ability to model complex phenomena. In this paper, we investigate the Cauchy problem for a nonlinear heat equation that incorporates both an integral nonlinearity and a continuous forcing term. The equation under consideration is given by:
\begin{equation} 
\begin{cases}
\partial_t u - \Delta u = \|u(t)\|_q^\alpha|u|^p + 
\,t^{\varrho}\,\mathbf{w}(x), & \text{in } \mathbb{R}^N \times (0, \infty), \\
u(x, 0) = u_0(x), & \text{in } \mathbb{R}^N,
\end{cases}
\label{main}
\end{equation}
where $N \geq 1 $ , $p > 1$, $q \geq 1$, $\alpha \geq 0$, $\varrho>-1$, and $\mathbf{w}: \mathbb{R}^N \to  \mathbb{R}$ is a continuous function.
The problem \eqref{main} arises in the theory of quasi-regular and quasi-conformal mappings, which have significant applications in modeling non-Newtonian fluid flow, population dynamics in biological systems, and other physical phenomena (see, e.g., \cite{Furter, Pao}). For a comprehensive overview of physical models derived from various fields and their mathematical analysis, we refer to \cite{Furter, Henry, MP92, Pao, Rothe, Hu, GV, QS, MT1, MT2, QS, Salsa, SGKM} and the references therein.

In the case \(\mathbf{w}(x) = 0\) and \(\alpha = 0\), problem \eqref{main} reduces to the classical Fujita-type equation:
\begin{equation}
\label{Fuj}
\left\{
\begin{matrix}
\partial_t u - \Delta u = |u|^{p}, & \text{in } \mathbb{R}^N \times (0, \infty), \\
u(x, 0) = u_0(x), & x \in \mathbb{R}^N.
\end{matrix}
\right.
\end{equation}
In their seminal works \cite{fujita, Fuj69}, Fujita studied the initial value problem \eqref{Fuj} for nonnegative initial data \(u_0(x) \geq 0\). He established the following fundamental results:
\begin{itemize}
    \item If \(p < 1 + \frac{2}{N}\), then {all nontrivial solutions} of \eqref{Fuj} {blow up in finite time}.
    \item If \(p > 1 + \frac{2}{N}\), then \eqref{Fuj} admits both bounded global solutions and solutions that blow up in finite time. Specifically, for initial data \(u_0(x)\) bounded by a sufficiently small Gaussian \(\varepsilon e^{-|x|^2}\), the solution exists globally. For further details, see \cite[Theorem 20.1, p. 129]{QS}.
\end{itemize}
In the borderline case \(p = 1 + \frac{2}{N}\), it was shown by Hayakawa \cite{Hayak} for dimensions \(N = 1, 2\) and by Kobayashi, Sino, and Tanaka \cite{KST77} and Aronson and Weinberger \cite{AW78} for all \(N \geq 1\) that all nontrivial solutions to \eqref{Fuj} blow up in finite time. The critical exponent \(p_F := 1 + \frac{2}{N}\) is now widely known as the \textit{Fujita exponent} for problem \eqref{Fuj}.
The case of a weighted nonlinearity of the form \(|x|^b |u|^p\) was investigated in \cite{Qi} (see also \cite{BL89, LM89}). The Cauchy problem \eqref{Fuj} is shown to have no global solution if $p<1+\frac{2+b}{N}$, but global solutions exist when $p>1+\frac{2+b}{N}$ provided that $b >-1$ for $N=1$ and $b>-2$ for $N\geq 2$. See \cite[Theorem 1.1, p. 125]{Qi} for a precise statement. However, the critical case $p=1+\frac{2+b}{N}$ is still open.

Note that the case $\alpha=\varrho=0$ was investigated by Bandle, Levine and Zhang in \cite{BLZ}. They showed that \eqref{main} does not have global solutions provided that $p<\frac{N}{N-2}$ and $\int_{\R^N}\,{\mathbf w}(x)\,dx>0$. 
 As noted in \cite{BLZ}, the exponent $\frac{N}{N-2}$ is closely tied to the critical exponent of the corresponding elliptic problem. This exponent is particularly significant because it is analogous to determining the smallest value of $p$ for which the equation $\Delta v + v^p = 0$ admits singular radial solutions of the form $v(r) = A r^{-\frac{2}{p-1}}$.

In \cite{JKS}, the authors consider \eqref{main} with $\alpha=0$ and  $\varrho >-1$. They showed that the critical exponent is given by
\begin{equation*}
\begin{split}
 p^*(\varrho)=\; \left\{
\begin{array}{cllll}\frac{N-2\varrho}{N-2\varrho-2}
\quad&\mbox{if}&\quad -1<\varrho<0,\\ \infty \quad
&\mbox{if}&\quad \varrho>0.
\end{array}
\right.
\end{split}
\end{equation*}
In \cite{MM}, the author considers the equation
\[
\partial_t u - \Delta u = |x|^\alpha |u|^p + \mathtt{a}(t) \mathbf{w}(x),
\]
where the function \(\mathtt{a}(t)\) exhibits the following asymptotic behavior:
\begin{equation*}
\mathtt{a}(t) \sim 
\begin{cases}
\ta_0\,t^\sigma & \text{as } t \to 0, \\
\ta_\infty\,t^m & \text{as } t \to \infty,
\end{cases}
\end{equation*}
with \(\ta_0, \ta_\infty > 0\), \(\sigma > -1\), and \(m \in \mathbb{R}\). 

It was shown that the critical exponent for this problem is given by
\begin{equation*}
p^*(m, \alpha) = 
\begin{cases}
\frac{N - 2m + \alpha}{N - 2m - 2} & \text{if } m \leq 0 \text{ and } \alpha > -2, \\
\infty & \text{if } m > 0 \text{ and } \alpha > -2,
\end{cases}
\end{equation*}
provided that \(\int_{\mathbb{R}^N} \mathbf{w}(x) \, dx > 0\).
Later, in \cite{Opuscula}, a broad class of forcing terms of the form \(\mathtt{a}(t) \mathbf{w}(x)\) is studied.
To the best of our knowledge, the first rigorous investigation of the critical Fujita exponent for diffusion equations with nonlocal nonlinearity was conducted by Galaktionov and Levine in \cite{GL1998}. They studied a class of equations of the form
\begin{equation}
    \label{GL}
    \partial_t u - \Delta u = \left(\int_{\mathbb{R}^N} K(y) u^q(y, t) \, dy \right)^{(p-1)/q} u^{r+1}, \quad x \in \mathbb{R}^N, \quad t > 0,
\end{equation}
where \(p \geq 1\), \(r \geq 0\), \(p + r > 1\), and \(1 \leq q < \infty\). The function \(K\) is assumed to be nonnegative and measurable, with additional technical assumptions. Using a subsolution approach, the authors determined the critical Fujita exponent under suitable conditions. Specifically, for the case \(K(y) = 1\), any nontrivial nonnegative solution blows up in finite time provided that 
\[
0 \leq r \leq \frac{2}{N} \quad \text{and} \quad p < 1 + \frac{2q(1 - Nr/2)}{N(q - 1)}.
\]
Later, the following nonlocal equation
\begin{equation}  
   \partial_t u-\Delta u=u^p\left(\displaystyle\int_{\R^N} u^q(y,t)dy\right)^{r}, \quad x \in \mathbb{R}^N, \quad t > 0, 
\label{Lap05}
\end{equation}
was considered in \cite{Lap2005}. Here, $p> 1$ and $q,r>0$. Among other results, it is established that if $p+(q-1)r \leqslant 1+\frac{2}{N}$, then all nontrivial nonnegative solutions blow-up in finite time. A similar problem was also studied in \cite{Lap2002}, namely

\begin{equation}  
   \partial_t u - \Delta u = u \left(\displaystyle\int_{\mathbb{R}^N} k(y) u^q(y, t) \, dy \right)^{p/q}, \quad x \in \mathbb{R}^N, \quad t > 0,
\label{Lap02}
\end{equation}
where \(p > 0\) and \(q \geq 1\). The kernel \(k(y)\) is assumed to be measurable, nonnegative, and satisfying the estimates
\begin{equation}
    \label{Kernel}
    (1 + |y|)^{\gamma} \lesssim k(y) \lesssim (1 + |y|)^{\gamma}, \quad \gamma \in \mathbb{R}.
\end{equation}

It was shown that every nontrivial nonnegative solution blows up in finite time if one of the following conditions holds:
\begin{itemize}
    \item \(0 \leq \gamma < N(q - 1)\) and \(0 < p \leq \frac{2q}{N(q - 1) - \gamma}\);
    \item \(N(q - 1) \leq \gamma\) and \(p > 0\).
\end{itemize}
It is tempting to conjecture that the {\it Fujita critical exponent} can be deduced from scaling considerations. Specifically, consider a nonlinear parabolic equation on $\mathbb{R}^N$ with the property that if $u(x,t)$ is a solution with initial data $u_0(x)$, then the rescaled function $\lambda^a u(\lambda x, \lambda^2 t)$, for $\lambda > 0$, is also a solution with initial data $\lambda^a u_0(\lambda x)$. This scaling invariance implies that the only Lebesgue space that preserves the norm of the rescaled initial data $\lambda^a u_0(\lambda x)$ is $L^{p_c}(\mathbb{R}^N)$, where $p_c = N/a$. Consequently, the Fujita critical exponent $p_F$ is expected to be determined by the condition $p_c = 1$, which corresponds to $a = N$. This scaling argument accurately predicts the value of $p_F$ for a broad class of nonlinear local diffusion equations \cite{Bandle89, BL89, Gal83, Gal80, Mei86, Mei88, Mei90}, as well as for certain nonlinear nonlocal parabolic problems. For example, it applies to equations of the form
$$
u_t = \Delta u + u(0,t)^{p-q} u(x,t)^q \quad \text{with} \quad p > q \geq 1,
$$
as investigated in \cite{Soup99}. However, there exist nonlinear nonlocal equations for which the Fujita exponent $p_F = \infty$ cannot be predicted by scaling arguments; see \cite{Soup98, Soup99}. A particularly interesting example is studied in \cite{CDW}, where the Fujita exponent is finite, but its value deviates from that expected based on scaling invariance. 

When applied to equation \begin{equation}
    \label{main-0}
    \partial_t u - \Delta u = \|u(t)\|_q^\alpha |u|^p,
\end{equation}
the scaling argument yields the Fujita exponent
\begin{equation}
    \label{Fuj-0}
    p + \left(1 - \frac{1}{q}\right)\alpha = 1 + \frac{2}{N}.
\end{equation}

The novelty in studying \eqref{main} lies in the interplay between the nonlocal nonlinearity $\|u(t)\|_q^\alpha |u|^p$ and the forcing term $t^{\varrho} \mathbf{w}(x)$. To the best of our knowledge, this is the first time such a problem has been considered.

It is worth noting that several well-established methods exist for studying the blow-up phenomenon, each with its specific domain of applicability to problems in mathematical physics. For a comprehensive survey of blow-up results for solutions of first-order nonlinear evolution inequalities and related Cauchy problems, we refer to \cite{BLZ, DL, GV} and the references therein. Further detailed discussions can be found in the books \cite{Hu, Pao, SGKM, QS}.

We adopt the following notion of global weak solution.
\begin{defi}
    \label{defn:weak-solution}
We say that $u$ is a global weak solution of \eqref{main} if it satisfies the conditions
\begin{equation}\label{eq:weak-solution}
\nonumber u_0\in L^1(\R^N),\quad u\in L^\infty_{loc}((0,\infty), L^q),\quad \|u(t)\|_q^{\alpha}|u|^{p}\in L^1_{loc}(\R^N\times (0,\infty))\end{equation} and
\begin{align*}
\int_{\R^N\times (0,\infty)}u(-\partial_{t}\psi-\Delta\psi)dx dt=\int_{\R^N}u_0(x)\psi(x,0)dx+\int_{\R^N\times (0,\infty)}\|u(t)\|_q^{\alpha}|u|^p\psi dxdt+\\
\qquad{}\hspace{3cm}\int_{\R^N\times (0,\infty)}t^{\varrho}\mathbf{w}\psi\,dxdt
\end{align*}
for all $\psi\in C^{\infty}_{0}(\R^N\times [0,\infty))$.
\end{defi}
As is standard practice, \eqref{main} is equivalent, within an appropriate framework, to the Duhamel formulation
\begin{equation}\label{Duhamel}
u(t)=\bs(t)u_0 +\int_0^t \bs(t-\tau)\|u(\tau)\|_{q}^\alpha |u(\tau)|^p  \, d\tau+\int_0^t  \tau^{\varrho} \bs(t-\tau) \bw\,d\tau
\end{equation}
where $\bs(t)=:{\rm e}^{t\Delta}$ denotes the linear heat semi-group. A solution of the integral equation \eqref{Duhamel} is commonly referred to as a {\it mild solution} of \eqref{main}. In our setting, it can be shown that a mild solution also constitutes a weak solution. For a detailed discussion on the equivalence between the differential and integral formulations for the nonlinear heat equation, we refer to \cite{Terr1, Terr2, RT, PGL}.
\subsection{Main results}
In what follows, we will make use of the following assumptions on $\bw$, which play a key role in applying the comparison principle and deriving blow-up results. Specifically, we assume that $\bw\in C(\R^N)\cap L^1(\R^N)$ satisfies either 
\begin{equation}
    \label{Ass-blow-w}
   \inf_{\lambda>0, x\in\R^N} \left(\int_{\R^N}\, {\rm e}^{-\frac{|x-y|^2}{\lambda}}\,\bw(y)dy\right)\geq 0,
\end{equation}
or
\begin{equation}
    \label{Ass-blow-ww}
    \int_{\R^N}\, \bw(x)dx> 0.
\end{equation}
The latter assumption is frequently used in the literature; see, for instance, \cite{Opuscula, BLZ, JKS, MM, Majd}.

Our first main result addresses the local well-posedness in Lebesgue spaces.
\begin{thm}
\label{LWP}
Let $N\geq 1$, $p>1$, $q\geq 1$, $\alpha\geq 1$, and $\varrho>-1$. Furthermore, suppose that $q>\frac{N(p-1)}{2}$.  Then the initial value problem \eqref{main} is locally well-posed in $L^q$. More precisely, given $u_0, \bw \in L^q(\mathbb{R}^N)$, then there exist $T>0$ and a unique mild solution $u\in C([0,T], L^q)$ to \eqref{main}.  
\end{thm}
\begin{rem}
    ~
{\rm 
    \begin{itemize}
    \item[a)] The restriction $\alpha \geq 1$ is required for the application of a fixed-point argument based on the inequality \eqref{key-contrac} below. While the case $\alpha = 0$ is straightforward, the case $\alpha \in (0,1)$ remains unresolved.
    \item[b)] The assumption $\varrho > -1$ is necessary to ensure that $t^\varrho \in L^1_{\text{loc}}(0,\infty)$.
\end{itemize} }
\end{rem}
When $q\geq p$, it is meaningful to consider the weak solution $u\in C([0,T], L^q(\R^N))$ in the integral sense as defined in \eqref{Duhamel}. Uniqueness of the solution is guaranteed within this class, as established in the following result.
\begin{thm}
\label{Uniq}
Let $ N \geq 1 $, $ p > 1 $, $ q \geq 1 $, $ \alpha \geq 1 $, and $ \varrho > -1 $. Assume the following conditions hold:
\begin{equation}
    \label{Cond-Uniq}
    q > N(p - 1)/2,\;\;q \geq p,\;\;\mathbf{w} \in L^q.
\end{equation}
Then, for any $ T > 0 $, the initial value problem \eqref{main} admits at most one solution in the space $ C([0, T], L^q) $.
\end{thm}
\begin{rem}
    {\rm Note that the case $\alpha = 0$ was studied in \cite[Theorem 4]{BC} for both $q > N(p - 1)/2$ (resp. $q = N(p - 1)/2$) and $q \geq p$ (resp. $q > p$). In this work, however, we are only able to address the case where $q > N(p - 1)/2$ and $q \geq p$.}
\end{rem}
{ In what follows we set
\begin{equation}
    \label{delt}
    \delta=\alpha\left(1-\frac{1}{q}\right).
\end{equation}
\begin{thm}
\label{Fujita-2-bis}
Let $ N \geq 3 $, $ p > 1 $, $ q \geq 1 $, $ \alpha \geq 0 $, and $ \varrho \leq 0 $. Suppose $ \delta < \frac{2}{N} $. Assume further that $ \mathbf{w} \in C(\mathbb{R}^N) \cap L^1(\mathbb{R}^N) $ satisfies the conditions \eqref{Ass-blow-w} and \eqref{Ass-blow-ww}. If the inequality
\begin{equation}
    \label{Fujita-exp}
    p + \frac{N\delta}{N - 2\varrho - 2} < \frac{N - 2\varrho}{N - 2\varrho - 2}
\end{equation}
holds, then the problem \eqref{main} admits no global weak solution with nonnegative initial data. 
\end{thm}
}
\begin{rem}
    {\rm When $\alpha = 0$, condition \eqref{Fujita-exp} reduces to $p < \frac{N - 2\varrho}{N - 2\varrho - 2}$, a result that was previously established in \cite{MM}.}
\end{rem}
We now turn to the analysis of the global theory. The fundamental result in this context is as follows:
\begin{thm}
\label{GEP}
Let $N\geq2$, $\delta<\frac{2}{N}$ and $-1<\varrho<0$. Assume that $p\geq\frac{N - 2\varrho-N\delta}{N - 2\varrho - 2}$, and, $ \frac{N(p-1)}{2} \leqslant q \leqslant p$. Let us define $\ell = \frac{N \left( (p-1)(q-1) + \delta \right)}{2(q-1) + N\delta + 2 (\varrho + 1) (p-1)(q-1) + 2 (\varrho + 1) \delta}$ and $p_c=\frac{N\left((p - 1)(q - 1) + \delta q\right)}{2(q - 1) + N\delta}$. Then, for any initial data  $u_0\in L^{p_c}(\mathbb{R}^N)$ and $\bw \in L^\ell(\mathbb{R}^N)$, with the property that $\|u_0\|_{L^{p_c}} + \|\bw\|_{L^\ell}$ is sufficiently small, the equation \eqref{main} admits a global solution $u$ in time.  
\end{thm}
The outline of the article is as follows. In the next section, we recall some basic facts and useful tools. Section \ref{S3} is devoted to local well-posedness and unconditional uniqueness, where the proofs of Theorem \ref{LWP} and Theorem \ref{Uniq} are provided. In Section \ref{S4}, we examine the nonexistence of global solutions and present the proof of Theorem \ref{Fujita-2-bis}. The global existence of solutions for sufficiently small initial data is addressed in Section \ref{S5}, where the proof of Theorem \ref{GEP} is given.

Throughout the paper, we adhere to the standard convention of denoting positive constants by $C$, which may vary from line to line. Additionally, we use the shorthand notation $\|\cdot\|_{r}$ to represent the Lebesgue norm $\|\cdot\|_{L^r(\mathbb{R}^N)}$ for $1 \leq r \leq \infty$.
\section{Useful tools \& Auxiliary results}
\label{S2}
In this section, we introduce key tools and auxiliary results that are essential for establishing our main results.  Let $\{{\mathbf S}(t)\}_{t\geq 0}$ be the linear heat semigroup defined by ${\mathbf S}(t)\,\varphi=G(\cdot,t)\star\varphi, t>0,$ where $G(\cdot,t)$ is the heat kernel given by
$$
G(x,t)=\left(4\pi t\right)^{-N/2}\,{\rm e}^{-|x|^2/(4t)},\;\;\;t>0,\;\;\;x\in\R^N.
$$
In the subsequent lemma, we gather the smmothing effect regarding the heat semigroup $\mathbf{S}(t)$ that will be helpful for our purposes.

\begin{lem}
\label{S-t}
Let $1\leq a\leq b\leq  \infty$, $t>0$, and $\varphi\in L^a(\R^N)$. Then 
\begin{equation}
       \label{smooth-effect}
    \|{\mathbf S}(t)\,\varphi\|_{L^b(\R^N)}\leq t^{-\frac{N}{2}(\frac{1}{a}-\frac{1}{b})}\|\varphi\|_{L^a(\R^N)}.
    \end{equation}
    \end{lem}

We now state a comparison principle for \eqref{main}. 
\begin{lem} \label{R-Maximum}
Let \( N \geq 1 \), \( p, q \geq 1 \), \( \alpha \geq 0 \), and \( \varrho > -1 \). Suppose \( u \) is a global mild solution to problem \eqref{main} with initial data \( 0 \leq u_0 \in L^1(\R^N)\), and assume that condition \eqref{Ass-blow-w} holds. Then, there exists a positive constant \( \mathtt{C}_0 \), depending only on \( N \) and \( u_0 \), such that the following estimate holds:
\begin{equation}
\label{comp-est}
 u(x,t)\geqslant \mathtt{C}_0 t^{-\frac{N}{2}} \exp\left({-\frac{|x|^2}{t}}\right),\quad t\geq 1,\, x\in\R^N.
\end{equation}    
\end{lem}
\begin{proof}
Let \( v \) satisfy the following problem:
\[
\begin{cases}
\partial_t v - \Delta v = t^{\varrho} \mathbf{w}(x), & (x, t) \in \mathbb{R}^N \times (0, \infty), \\
v(x, 0) = u_0(x), & x \in \mathbb{R}^N.
\end{cases}
\]
The solution \( v \) can be expressed as
\[
v(x, t) = \int_{\mathbb{R}^N} G(x - y, t) u_0(y) \, dy + \int_0^t \int_{\mathbb{R}^N} G(x - y, t - \tau) \tau^{\varrho} \mathbf{w}(y) \, dy \, d\tau,
\]
where \( G \) is the heat kernel. Using \eqref{Duhamel} and the assumption \eqref{Ass-blow-w}, we observe that for \( t \geq 1 \),
\[
\begin{split}
u(x, t) &\geq v(x, t) \\
&\geq \int_{\mathbb{R}^N} G(x - y, t) u_0(y) \, dy \\
&= \int_{\mathbb{R}^N} (4\pi t)^{-\frac{N}{2}} \exp\left(-\frac{|x - y|^2}{4t}\right) u_0(y) \, dy \\
&\geq (4\pi t)^{-N/2} \exp\left(-\frac{|x|^2}{2t}\right) \int_{\mathbb{R}^N} \exp\left(-\frac{|y|^2}{2}\right) u_0(y) \, dy.
\end{split}
\]
Finally, we obtain the desired estimate \eqref{comp-est} with the constant
\[
\mathtt{C}_0 = (4\pi)^{-N/2} \int_{\mathbb{R}^N} \exp\left(-\frac{|y|^2}{2}\right) u_0(y) \, dy.
\]
\end{proof}
\begin{rem}
    {\rm A direct consequence of \eqref{comp-est} is that, for any  $1 \leq q < \infty $, the following estimate holds:
    \begin{equation}
        \label{comp-est-q}
        \|u(t)\|_q\geq C t^{-\frac{N}{2}\left(1-\frac{1}{q}\right)},\quad t\geq 1,
    \end{equation}
    where the constant $C$ is given by $C=\pi^{N/2q}q^{-N/2q} \mathtt{C}_0.$}
\end{rem}
The comparison of the two conditions \eqref{Ass-blow-w} and \eqref{Ass-blow-ww} is discussed in the following statement.
\begin{lem}
    \label{Assumptions-w}
Let $\bw \in C(\R^N) \cap L^1(\R^N)$. The following hold:
\begin{enumerate}
    \item[(i)] The assertion \eqref{Ass-blow-w} implies $\int_{\R^N}\, \bw(x)dx\geq 0$.
    \item[(ii)] There exists a function $\bw \in C(\R^N) \cap L^1(\R^N)$ satisfying \eqref{Ass-blow-ww} and the additional condition:
    \begin{equation}
        \label{Counter-examp}
        \int_{\R^N} {\rm e}^{-|y|^2} \bw(y) \, dy \leq 0.
    \end{equation}
    This demonstrates that \eqref{Ass-blow-ww} does not, in general, imply \eqref{Ass-blow-w}.
\end{enumerate}
\end{lem}
\begin{proof}
The proof of part (i) follows by taking the limit in \eqref{Ass-blow-w} as $\lambda\to\infty$ and applying the Lebesgue Dominated Convergence Theorem. \\
For part (ii), consider the following simple example 
\begin{equation}
    \label{examp}
    \bw(y) = a\,{\rm e}^{-|y|^2} - {\rm e}^{-2|y|^2}, \quad 2^{-N/2}< a \leq (\frac{2}{3})^{N/2}.
\end{equation}
To analyze this, we use the well-known result for Gaussian integrals:
\begin{equation}
\label{Gauss-1}
    \int_{\mathbb{R}^N} {\rm e}^{-\theta |y|^2} dy = \left(\frac{\pi}{\theta}\right)^{N/2}, \quad \theta > 0.
\end{equation}
Applying \eqref{Gauss-1}, we compute the following integrals:
\begin{equation}
 \label{Gauss-2}   
\begin{split}
    \int_{\mathbb{R}^N} \bw(y) \, dy &= \pi^{N/2} \left(a - 2^{-N/2}\right), \\
    \int_{\mathbb{R}} {\rm e}^{-y^2} \bw(y) \, dy &= \pi^{N/2} \left(a2^{-N/2} - 3^{-N/2}\right).
\end{split}
\end{equation}
The conclusion follows directly from \eqref{Gauss-2}.
\end{proof}

We also recall the well known $\varepsilon-$Young inequality.
\begin{lem}
\label{Young}
Let $a, b\geq 0$, $p, q>1$ and $\varepsilon>0$ such that $\frac{1}{p}+\frac{1}{q}=1.$
Then
\begin{equation}
\label{Young-Ineq}
ab\leq \varepsilon\,a^p
+(p\varepsilon)^{-q/p}\,\frac{b^q}{q}.
\end{equation}
\end{lem}

When applying the test-function method to analyze the blow-up, we will utilize the following result.
\begin{lem}
    \label{lapl-g}
    Let $ g \in C^\infty([0, \infty)) $ be a smooth, nonnegative function such that $ g(\tau) = 0 $ for all $ \tau \geq \tau_0 $, where $ \tau_0 > 0 $ is a fixed constant. For $ T > 0 $ and $ \theta > 2 $, define the function $ g_{T,\theta}: \mathbb{R}^N \to \mathbb{R} $ by
\begin{equation}
      \label{g-T}
      g_{T,\theta}(x)=\left(g\left(\frac{|x|^2}{T}\right)\right)^{\theta}.
    \end{equation}
Then, there exists a constant $ C > 0 $, depending only on $ g $, $ \theta $, and the dimension $ N $, such that
\begin{equation}
        \label{Lap-g-est}
        \Big|\Delta\,g_{T,\theta}(x)\Big|\leq \frac{C}{T}\left(g\left(\frac{|x|^2}{T}\right)\right)^{\theta-2}.
    \end{equation}
where $ \Delta $ denotes the Laplacian operator.
\end{lem}
\begin{proof}
By performing the change of variable $ y = \frac{|x|^2}{T} $, we express  
$$
g_{T,\theta}(x) = g^\theta(y) = v(r), \quad \text{where } r = |x|.
$$
To compute the Laplacian, we utilize the well-known formula for radially symmetric functions:  
\begin{equation}
\label{Lap-Rad}
\Delta v = v'' + \frac{N-1}{r} v'.
\end{equation}
Applying this formula, we derive  
\begin{equation}
\label{Lap-g-theta}
\Delta\,g_{T,\theta}(x) = \Bigg(2\theta N g'(y)g(y) + 4\theta y g''(y)g(y) + 4\theta(\theta-1)y g'(y)^2\Bigg) \frac{g(y)^{\theta-2}}{T}.  
\end{equation}
Since $ g $ has compact support, its derivatives $ g(y) $, $ g'(y) $, and $ g''(y) $ are bounded. Furthermore, since $ y \leq \tau_0 $, it follows from \eqref{Lap-g-theta} that  
$$
|\Delta g^{\theta}(y)| \leq \frac{C}{T} g(y)^{\theta - 2},
$$
where the positive constant $ C $ depends on $ N $, $ \theta $, $ \tau_0 $, $ \|g\|_{L^\infty} $, $ \|g'\|_{L^\infty} $, and $ \|g''\|_{L^\infty} $.  
This completes the proof.
\end{proof}
To establish well-posedness and apply a fixed-point argument, we will utilize the following estimates, which can be easily derived.
\begin{lem}
    \label{Contrac-estim}
    Let $(E, \|\cdot\|)$ be a normed vector space, $p>1$, $a, b\geq 0$, $\alpha\geq 0$, and $x,y\in E$. The following inequalities hold
    \begin{equation}
    \label{key-contrac}
         \Big|a^p\|x\|^\alpha-b^p\|y\|^\alpha\Big|\lesssim  \left\{\begin{array}{ll} \|y\|^\alpha|a-b|(a^{p-1}+b^{p-1})+a^p\|x-y\|(\|x\|^{\alpha-1}+\|y\|^{\alpha-1}) & \text{if}\,\, \alpha\geq 1\\ \|y\|^\alpha|a-b|(a^{p-1}+b^{p-1})+a^p\|x-y\|^\alpha  & \text{if}\,\, \alpha<1\end{array} \right. 
    \end{equation}
\end{lem}
The following classical result will be employed to establish the desired regularity for the local solution of \eqref{main} in the $L^q$ space.
\begin{lem}
    \label{C-L-q}
    Let $1\leq r <\infty$, $T>0$, and $F\in L^1((0,T), L^r(\R^N))$. Then,
    \begin{equation}
        \label{v-CLq}
        \mathbf{v}(t)=: \int_0^t\,{\rm e}^{(t-s)\Delta}\,F(s)\,ds \in \text{C}([0,T], L^r(\R^N)).
    \end{equation}
\end{lem}
\begin{proof}
For completeness, we provide a proof of the above lemma. Fix $ 0 \leq t_0 \leq T $, and write
$$
\mathbf{v}(t) - \mathbf{v}(t_0) = \int_{t_0}^t {\rm e}^{(t-s)\Delta} F(s) \, ds + \int_0^{t_0} {\rm e}^{(t_0-s)\Delta} \left( {\rm e}^{(t-t_0)\Delta} F(s) - F(s) \right) \, ds.
$$
Using \eqref{smooth-effect}, it follows that
\begin{equation}
    \label{Regu-CLq}
    \|\mathbf{v}(t)-\mathbf{v}(t_0)\|_r\leq \int_{t_0}^t \|F(s)\|_r ds+\int_0^{t_0}\|{\rm e}^{(t-t_0)\Delta}F(s)-F(s) \|_r ds.
\end{equation}
By the dominated convergence theorem and the fact that $ {\rm e}^{t\Delta}\phi \to \phi $ as $ t \to 0 $ in $ L^r $, we deduce from \eqref{Regu-CLq} that $ \mathbf{v} $ is continuous at $ t_0 $.
\end{proof}
We also make use of the following singular Gronwall inequality \cite[Theorem 3.1, p. 537]{DM1986}
\begin{lem}
\label{SGI}
Let $T>0$ and $0\leq \psi\in C([0,T])$ satisfy
$$
\psi(t)\leq A+M\displaystyle\int_0^t (t-\sigma)^{-\sigma}\psi(\sigma)\,d\tau, \quad \quad 
0\leq t\leq T,
$$
where $0\leq \sigma<1$, {$A\geq 0$, and $M>0$}. Then, 
\begin{equation}
\label{SGI2}
\psi(t)\leq A\, {\mathcal E}_{1-\sigma}\left(M \Gamma(1-\sigma)\,t^{1-\sigma}\right),\quad 0\leq t\leq T,
\end{equation}
where ${\mathcal E}_{1-\sigma}$ is the Mittag-Leffler function 
defined by
$$
{\mathcal E}_{1-\sigma}(z)=\sum_{n=0}^{\infty}\,\frac{z^n}{\Gamma(n(1-\sigma)+1)}.
$$
In particular, if $A=0$ then $\psi\equiv 0$.
\end{lem}

\section{Local well-posedness and unconditional uniqueness}
\label{S3}
\subsection{Local well-posedness}

To prove Theorem \ref{LWP}, we use the contraction mapping principle in a complete metric space.
Fix $M>\|u_0\|_q$. For $T>0$, define
$$
\mathbf{X}(T) = \left\{ u \in  L^\infty((0,T), L^q) \cap L^\infty_{loc}((0,T), L^{pq})
;\;
\sup_{0<t<T} \|u(t)\|_{L^q}\leq M\;\mbox{and}\; \sup_{0<t<T}t^{\sigma}\|u(t)\|_{pq}  \leq M \right\},
$$
where 
\begin{equation}
    \label{sigma}
    \sigma=\frac{N(p-1)}{2pq}.
\end{equation} We endow $\mathbf{X}(T)$  with the distance 
\begin{equation*}
d(u, v) = \sup_{0<t<T}\|u(t)-v(t)\|_{q} +\sup_{0<t<T}t^{\sigma}\|u(t)-v(t)\|_{pq}.
\end{equation*}
It is clear that  $\left(\mathbf{X}(T), d\right)$ is a  nonempty complete metric space.
\begin{proof}[Proof of Theorem \ref{LWP}]  

Given $u \in\mathbf{X}(T)$, we define

\begin{equation}
    \label{Phi}
    \Phi(u)(t)=\bs(t)u_0 +\int_0^t \bs(t-\tau)\|u(\tau)\|_{q}^\alpha |u(\tau)|^p  \, d\tau+\int_0^t  \tau^{\varrho} \bs(t-\tau) \bw\,d\tau.
\end{equation}
Our goal is to show that, for a suitably chosen $T$, the map $\Phi$ is well-defined from $\mathbf{X}(T)$ into itself and constitutes a contraction.

For $u \in \mathbf{X}(T)$ and $t \in (0, T)$, it follows from \eqref{smooth-effect} that
\begin{equation}
\label{Stab1}
\begin{split}
\|\Phi(u)(t)\|_q&\leq \|u_0\|_q+M^{\alpha+p}\int_0^t \tau^{-p\sigma} d\tau+\frac{T^{\varrho+1}}{\varrho+1}\|\bw\|_q\\
&\leq\|u_0\|_q+\frac{T^{1-p\sigma}}{1-p\sigma}M^{\alpha+p}+\frac{T^{\varrho+1}}{\varrho+1}\|\bw\|_q.
\end{split}
\end{equation}
Similarly, it follows that
\begin{equation}
\label{Stab2}
\begin{split}
\|\Phi(u)(t)\|_{pq}
&\leq t^{-\sigma}\|u_0\|_q+M^{\alpha}\int_0^t (t-\tau)^{-\sigma} \|u(\tau)\|_{pq}^p\,d\tau+\left(\int_0^t\,\tau^{\varrho}(t-\tau)^{-\sigma}\,d\tau\right)\|\bw\|_q\\
&\leq t^{-\sigma}\|u_0\|_q+M^{\alpha+p}t^{1-\sigma-p\sigma}\left(\int_0^1\,s^{-p\sigma}(1-s)^{-\sigma}\,ds\right) +t^{\varrho+1-\sigma}\|\bw\|_q\left(\int_0^1\,s^{\varrho}(1-s)^{-\sigma}\,ds\right)\\
&\leq t^{-\sigma}\bigg(\|u_0\|_q+ M^{\alpha+p}T^{1-p\sigma}{\mb}(1-p\sigma, 1-\sigma)+T^{\varrho+1}\|\bw\|_q {\mb}(1+\varrho, 1-\sigma)\bigg),
\end{split}
\end{equation}
where $\mb$ denotes the well-known Beta function. Here, we have used the condition $\varrho > -1$ and the straightforward observation that $\sigma < \frac{1}{p} < 1$ where $\sigma$ is given by \eqref{sigma}.

From \eqref{Stab1} and \eqref{Stab2}, it follows easily that $\Phi(u) \in \mathbf{X}(T)$, provided $T > 0$ is chosen sufficiently small such that
\begin{equation}
\label{Stab3}
\begin{split}
    \frac{T^{1-p\sigma}}{1-p\sigma}M^{\alpha+p}+\frac{T^{\varrho+1}}{\varrho+1}\|\bw\|_q&<M-\|u_0\|_q\\
    M^{\alpha+p}T^{1-p\sigma}{\mb}(1-p\sigma, 1-\sigma)+T^{\varrho+1}\|\bw\|_q {\mb}(1+\varrho, 1-\sigma)&<M-\|u_0\|_q
\end{split}
\end{equation}

The following step is to show that $\Phi$ is a contraction. To verify this, let us consider $u, v \in \mathbf{X}(T)$, and write using \eqref{smooth-effect} together with \eqref{key-contrac} 
\begin{equation}
    \label{Contra1}
    \begin{split}
        \left\| \Phi(u)(t) - \Phi(v)(t) \right\|_{{q}} &\leqslant C \int_0^t \left\| u(\tau) \right\|_{pq}^p \left\| u(\tau) - v(\tau) \right\|_{q}  \left( \left\| u(\tau) \right\|_{q}^{\alpha - 1} + \left\| v(\tau) \right\|_{q}^{\alpha - 1} \right) d\tau \\ &\quad + C \int_0^t \left\| v(\tau) \right\|_{q}^\alpha \left\|\left| u(\tau) - v(\tau) \right| \left( \left| u(\tau) \right|^{p-1} + \left| v(\tau) \right|^{p-1} \right) \right\|_{q} d\tau\\
        &\leq C M^{\alpha+p-1} T^{1-p\sigma} d(u,v),
    \end{split}
\end{equation}
where we have used H\"older's inequality with $\frac{1}{q}=\frac{1}{pq}+\frac{p-1}{pq}$.

On the other hand, by using the smoothing effect inequality \eqref{smooth-effect}, and taking into account \eqref{sigma} and \eqref{key-contrac},  we get
\begin{equation}
    \label{Contra2}
    \begin{split}
\left\| \Phi(u)(t) - \Phi(v)(t) \right\|_{pq} &\leqslant C \int_0^t (t-\tau)^{-\sigma}\left\| u(\tau) \right\|_{pq}^p \left\| u(\tau) - v(\tau) \right\|_{q}  \left( \left\| u(\tau) \right\|_{q}^{\alpha - 1} + \left\| v(\tau) \right\|_{q}^{\alpha - 1} \right) d\tau \\ &\quad + C \int_0^t (t-\tau)^{-\sigma}\left\| v(\tau) \right\|_{q}^\alpha \left\|\left| u(\tau) - v(\tau) \right| \left( \left| u(\tau) \right|^{p-1} + \left| v(\tau) \right|^{p-1} \right) \right\|_{q} d\tau\\
        &\leq C M^{\alpha+p-1} t^{1-p\sigma-\sigma} d(u,v).
        \end{split}
\end{equation}
It follows from \eqref{Contra1}
and \eqref{Contra2} that
 \begin{equation}
     \label{Contra3}
  d\left(\Phi(u),\Phi(v)\right)\leqslant C M^{\alpha+p-1} T^{1-p\sigma} d(u,v).
 \end{equation}
It follows from \eqref{Contra3} that for $ T > 0 $ sufficiently small, the mapping $ \Phi: \mathbf{X}(T) \to \mathbf{X}(T) $ is a contraction. By the Banach fixed-point theorem, $ \Phi $ admits a unique fixed point $ u \in \mathbf{X}(T) $, which constitutes a solution to \eqref{Duhamel}. 

To complete the proof, it remains to establish that $ u $ possesses the regularity $ u \in C([0, T], L^q) $. Since $ u \in \mathbf{X}(T) $, and under the assumptions $ p\sigma < 1 $, $ \mathbf{w} \in L^q $, and $ \varrho > -1 $, it follows that the term $ \|u(t)\|_q^\alpha |u|^p + t^{\varrho} \mathbf{w}(x) $ belongs to $ L^1((0, T), L^q) $. Invoking Lemma \ref{C-L-q}, we conclude that $ u \in C([0, T], L^q) $, as desired.
\end{proof}
\subsection{Unconditional uniqueness}
Our goal is to prove the {unconditional uniqueness} stated in Theorem \ref{Uniq}. To this end, let $ T > 0 $, and let $ u, v \in C([0, T], L^q) $ be two solutions of \eqref{main}. Define
$$
\psi(t) = \|u(t) - v(t)\|_q \quad \text{and} \quad \mathcal{M} = \mathcal{M}(T) = \sup_{0 \leq t \leq T} \left( \|u(t)\|_q + \|v(t)\|_q \right).
$$
From \eqref{Duhamel}, \eqref{key-contrac}, and \eqref{smooth-effect}, we derive the inequality
$$
\psi(t) \leq C \mathcal{M}^{\alpha + p - 1} \int_0^t (t - \tau)^{-\sigma} \psi(\tau) \, d\tau,
$$
where $\sigma$ is given by \eqref{sigma}. Invoking Lemma \ref{SGI}, we conclude that $ \psi(t) \equiv 0 $ for all $ t \in [0, T] $, which implies $ u = v $ in $ C([0, T], L^q) $. This establishes the desired uniqueness.
\section{Nonexistence of global solutions}
\label{S4}
Our goal is to prove the nonexistence of global weak solutions under the assumptions stated in Theorem \ref{Fujita-2-bis}. To achieve this, we employ the test function method, following the approach used in \cite{Gued-Kira, JKS, MM, Majd,  ES} for instance.
 \begin{proof}[Proof of Theorem \ref{Fujita-2-bis}]
 Suppose that the conditions of Theorem \ref{Fujita-2-bis} hold, and assume that
\begin{equation}
    \label{Contrad-0}
    p + \frac{N\delta}{N - 2\varrho - 2} < \frac{N - 2\varrho}{N - 2\varrho - 2}.
\end{equation}
We proceed by contradiction, assuming that the problem \eqref{main} admits a global weak solution with nonnegative initial data in the sense of Definition \ref{defn:weak-solution}. To derive a contradiction, we employ the test function method. Specifically, we introduce carefully chosen cut-off functions
$\psi_k\in C^{\infty}_{0}([0,\infty))$, $k=1,2$ with $0\leq \psi_k\leq 1$ and
$$\psi_1(s)=\begin{cases}
1 \hspace{0.12cm}\text{if}\hspace{0.12cm} 1/2\leq s\leq 3/4\\
0 \hspace{0.12cm}\text{if}\hspace{0.12cm} s\in [0,1/4]\cup[4/5,\infty)
\end{cases},\quad\psi_2(s)=\begin{cases}
1 \hspace{0.12cm}\text{if}\hspace{0.12cm}  s\in [0,1]\\
0 \hspace{0.12cm}\text{if}\hspace{0.12cm} s\geq 2.
\end{cases}
$$  	
Next, for $T>1$ sufficiently large, we define the function
$$\psi_{T}(x,t)=\psi_1\bigg(\frac{t-1}{T}\bigg)^{\frac{p}{p-1}}\psi_2\bigg(\frac{|x|^{2}}{T}\bigg)^{\frac{2p}{p-1}}.
$$
By multiplying both sides of the equation \eqref{main} by $\psi_T$ and integrating over the domain  $(1, T) \times \mathbb{R}^N$, we obtain
\begin{equation}
\label{Contad-1}
\begin{split}
     -\int_1^T \int_{\mathbb{R}^N} u \, \partial_t \psi_T \, dx \, dt 
-\int_1^T \int_{\mathbb{R}^N} u \, \Delta \psi_T \, dx \, dt&= \int_1^T \int_{\mathbb{R}^N} \|u\|_{q}^\alpha |u|^p \, \psi_T \, dx \, dt 
\\&+\int_1^T \int_{\mathbb{R}^N} t^{\varrho} \bw(x) \, \psi_T \, dx \, dt 
\\&+ \int_{\mathbb{R}^N} u_1(x) \, \psi_T(x,1) \, dx.
\end{split}
\end{equation}
Since $\psi_1(x,1) = 0$, it follows that 
$$
\int_{\mathbb{R}^N} u_{1}(x) \psi_T(x,1) dx = 0.
$$
Moreover, from \eqref{Contad-1}, we deduce that

\begin{equation}
    \label{Contrad-2}
\mathbb{K}(T)+\int_1^T\int_{\mathbb{R}^N}  t^{\varrho}\bw(x) \psi_T d x d t\leq \mathbb{I}(T)+\mathbb{J}(T)
\end{equation}
where
\begin{eqnarray}
\label{I}
\mathbb{I}(T)&=& \int_1^T\int_{\mathbb{R}^N}|u| |\Delta \psi_T| d x d t, \\
\label{J}
\mathbb{J}(T)&=&\int_1^T\int_{\mathbb{R}^N}|u|\left|\partial_t \psi_T\right| d x d t, \\
\label{K}
\mathbb{K}(T)&=&C\int_1^T\int_{\mathbb{R}^N} t^{-\frac{N \alpha}{2}(1-\frac{1}{q})} |u|^p \psi_T d x d t.
\end{eqnarray}
Here, we have used \eqref{comp-est-q} for the last term $\mathbb{K}(T)$.

Applying the $\varepsilon-$Young inequality \eqref{Young-Ineq}, we obtain
\begin{equation}
    \label{Contrad-3}
    \begin{split}
 \mathbb{I}(T)&=\int_1^{T} \int_{\mathbb{R}^N}\left(t^{-\frac{N}{2 p} \alpha\left(1-\frac{1}{q}\right)}|u|  
\psi_T^{\frac{1}{p}}\right) \left(t^{\frac{N}{2p} \alpha\left(1-\frac{1}{q}\right)} \psi_T^{-\frac{1}{p}} \left|\Delta \psi_T\right|\right)d x d t \\
& \leq \frac{1}{4}\mathbb{K}(T)+ C \underset{\mathbb{I}_1(T)}{\underbrace{\int_1^T \int_{\mathbb{R}^N} {t^{\frac{N}{2} \frac{\alpha}{p-1}\left(1-\frac{1}{q}\right)}} \psi_T^{-\frac{1}{p-1}}\left|\Delta \psi_T\right|^{\frac{p}{p-1}} d x d t}}.
    \end{split}
\end{equation}

To bound the term $\mathbb{I}_1(T)$, we apply Lemma \ref{lapl-g} after performing the change of variable $\tau = \frac{t-1}{T}$. This leads to
\begin{equation}
    \label{Contrad-4}
  \begin{split}
      \mathbb{I}_1(T)&\leq C T^{-\frac{p}{p-1}}\left(\int_1^T t^{\frac{N}{2} \frac{\alpha}{p-1}\left(1-\frac{1}{q}\right)} \psi_1\left(\frac{t-1}{T}\right)^{\frac{p}{p-1}} dt \right) {\left( \int_{|x|<\sqrt{2T}} dx\right)}\\
      &\leq C T^{1+\frac{N}{2}-\frac{p}{p-1}} \left(\int_{0}^{1} (1 + T\tau)^{\frac{N}{2} \frac{\alpha}{p-1}\left(1-\frac{1}{q}\right)} \psi_1(\tau)^{\frac{p}{p-1}} d\tau \right)\\
      &\leq C T^{1+\frac{N}{2}-\frac{p}{p-1}+\frac{N\delta}{2(p-1)}},
  \end{split}  
\end{equation}
where $\delta$ is given by \eqref{delt}. Therefore, from \eqref{Contrad-3}, we obtain
\begin{equation}
    \label{Contrad-5}
    \mathbb{I}(T) \leq \frac{1}{4}\mathbb{K}(T) + C T^{1+\frac{N}{2}-\frac{p}{p-1}+\frac{N\delta}{2(p-1)}}.
\end{equation}

Similarly, we find that
\begin{equation}
    \label{Contrad-6}
    \mathbb{J}(T) \leq \frac{1}{4}\mathbb{K}(T) + C T^{1+\frac{N}{2}-\frac{p}{p-1}+\frac{N\delta}{2(p-1)}}.
\end{equation}

Returning to \eqref{Contrad-2} and combining the estimates \eqref{Contrad-5} and \eqref{Contrad-6}, we obtain
\begin{equation}
    \label{Contrad-7}
    \int_1^T \int_{\mathbb{R}^N} t^\varrho \bw(x) \psi_T(x) \, dx \, dt \leq C T^{1+\frac{N}{2}-\frac{p}{p-1}+\frac{N\delta}{2(p-1)}}.
\end{equation}

To complete the proof, we establish a suitable lower bound for the left-hand side of \eqref{Contrad-7}. Writing
\begin{equation}
    \label{Contrad-8}
    \begin{split}
        \int_1^T \int_{\mathbb{R}^N} t^\varrho \bw(x) \psi_T(x) \, dx \, dt &= \left(\int_1^T t^{\varrho} \psi_1\left(\frac{t-1}{T}\right)^{\frac{p}{p-1}} \, dt\right) \left(\int_{\mathbb{R}^N} \psi_2\left(\frac{|x|^2}{T}\right)^{\frac{2p}{p-1}} \bw(x) \, dx\right) \\
        &\geq C T^{\varrho+1} \int_{\mathbb{R}^N} \bw(x) \psi_2\left(\frac{|x|^2}{T}\right)^{\frac{2p}{p-1}} \, dx,
    \end{split}
\end{equation}
we observe that since $\bw \in L^1(\mathbb{R}^N)$ and $\psi_2(0) = 1$, the dominated convergence theorem implies
$$
\lim_{T \to \infty} \int_{\mathbb{R}^N} \bw(x) \left(\psi_2\left(\frac{|x|^2}{T}\right)\right)^{\frac{2p}{p-1}} \, dx = \int_{\mathbb{R}^N} \bw(x) \, dx > 0.
$$
Thus, for sufficiently large $T > 1$, we have
\begin{equation}
    \label{Contrad-9}
    \int_{\mathbb{R}^N} \bw(x) \, dx \leq C T^{\frac{N}{2}-\frac{p}{p-1}+\frac{N\delta}{2(p-1)}-\varrho}.
\end{equation}

Noting that
\begin{equation}
    \label{Contrad-10}
    \frac{N}{2} - \frac{p}{p-1} + \frac{N\delta}{2(p-1)} - \varrho = \frac{N - 2\varrho - 2}{2} + \frac{N\delta - 2}{2(p-1)},
\end{equation}
and recalling that $N\delta - 2 < 0$ and \eqref{Contrad-0}, we easily see that
\begin{equation}
    \label{Contrad-11}
    \frac{N - 2\varrho - 2}{2} + \frac{N\delta - 2}{2(p-1)} < 0.
\end{equation}

Letting $T \to \infty$ in \eqref{Contrad-9}, we conclude that
$$
\int_{\mathbb{R}^N} \bw(x) \, dx \leq 0,
$$
which is a contradiction. This completes the proof of Theorem \ref{Fujita-2-bis}.\end{proof}

\section{Global existence}
\label{S5}
The following lemma plays a crucial role in the proof of Theorem \ref{GEP}.
\begin{lem} \label{tool}
Let $N\geq2$, $p,q>1$, $\alpha \geq 1$, and $-1<\varrho < 0$. Suppose that $ 0\leqslant \delta < \frac{2}{N} $, if
\begin{equation}
p\geq\frac{N - 2\varrho-N\delta}{N - 2\varrho - 2}
\label{E1}  
\end{equation}
Then, the following inequalities hold: 
\begin{equation}
\frac{2(q-1)+N\delta p}{Np\left((p-1)(q-1)+q\delta\right)}<\frac{1}{p_c},
\label{E2}    
\end{equation}
\begin{equation}
\frac{2(q-1)+N\delta p}{Np\left((p-1)(q-1)+q\delta\right)}< \frac{(p-1)(q+\delta-1)}{p\left[(q-1)(p-1)+q\delta \right]},
\label{E3}    
\end{equation}
\begin{equation}  
\frac{1}{p_c}+\frac{2\varrho}{N}+\frac{(p-1)[2Nqp-Nq\delta(2-\delta)]-N\delta-2pq\delta}{Np(p-1)[(p-1)(q-1)+ q\delta]}<\frac{N(p-1)(q+\delta-1)}{Np\left[(q-1)(p-1)+q\delta \right]},
\tag{E4}
\label{E4}  
\end{equation}
where $\delta$ is defined in \eqref{delt} and $p_c$ is as given in Theorem \ref{GEP}.
\end{lem}
If we further consider the case \( \frac{N(p-1)}{2} \leq q \leq p \) and apply Lemma \ref{tool}, we can choose \( r \geq 1 \) satisfying:
\begin{equation}  
\label{choice-r}
\begin{split}
   \max\left(\frac{2(q-1)+N\delta p}{Np\left((p-1)(q-1)+q\delta\right)}, {\frac{1}{p_c} + \frac{2\varrho}{N} + \frac{2\delta}{Np\left[(q-1)(p-1)+q\delta \right]}}\right) &< \frac{1}{r} \\
   \min \left(\frac{1}{p_c}, \frac{(p-1)(q+\delta-1)}{p\left[(q-1)(p-1)+q\delta \right]}\right)&>\frac{1}{r}
\end{split}
\end{equation}
With this choice, we observe that \( 1 \leq l < p_c < r \leq \infty \). Assuming \( r \) is fixed, we set
\begin{equation}
\label{choice- beta }
\beta =\frac{N}{2}\left(\frac{1}{p_c}-\frac{1}{r}\right)= \frac{2r(q - 1) + N\delta r - N(p - 1)(q - 1) - N\delta q}{2r \left((p - 1)(q - 1) + \delta q\right)}.
\end{equation}
It is straightforward to verify that 
\[
0 < \beta < \frac{1}{p+\alpha} = \frac{q-1}{p(q-1) + q\delta}.
\]
Let \(\nu\) be a sufficiently small positive constant. We define the function space \(\mathbf{Y}\) as follows:
$$
\mathbf{Y} = \left\{ u \in  L^\infty((0,\infty),L^r(\mathbb{R}^N)) \; ; \: 
 \sup_{t>0} t^{\beta}\|u(t)\|_{r} \leqslant \nu \right\}.$$
Equipped with the distance 
\[
d(u, v) = \sup_{t > 0} \, t^{\beta} \|u - v\|_{r},
\]
it is clear that \((\mathbf{Y}, d)\) forms a complete metric space.

Now, assume \( p \geq \frac{N - 2\varrho - N\delta}{N - 2\varrho - 2} \). Let \( u_0 \in L^{q_c} \) and \( \mathbf{w} \in L^{\ell} \) satisfy 
\[
\|u_0\|_{L^{p_c}} + \|\mathbf{w}\|_{L^\ell} < \epsilon
\]
for some \( \epsilon > 0 \). We will show that the map \(\Phi\) defined in \eqref{Phi} admits a unique fixed point in \(\mathbf{Y}\).

Since \( u_0 \in L^{p_c} \), it follows from \eqref{choice-r} and \eqref{smooth-effect} that
\begin{equation}
\| e^{t \Delta} u_0 \|_r \leq t^{-\frac{N}{2} \left( \frac{1}{p_c} - \frac{1}{r} \right)} \| u_0 \|_{p_c} = t^{-\beta} \| u_0 \|_{p_c}, \quad t > 0  
 \end{equation}
Moreover, as a consequence of \eqref{choice-r}, we have \( r > p \). Under the given hypothesis on \( q \), applying \eqref{smooth-effect} yields
\begin{align}
\int_0^t \| e^{(t-\tau) \Delta} |u(\tau)|^p \|u(\tau)\|_{q}^\alpha \|_{r}  \, d \tau &\leqslant  \int_0^t (t-\tau)^{-\frac{N}{2r} (p-1)} \| u(\tau) \|_{r}^p \| u(\tau) \|_{q}^\alpha \, d \tau \\
&\leqslant  \int_0^t (t-\tau)^{-\frac{N}{2r} (p-1)} \tau^{-\beta\alpha}(\tau\|u\|_{q})^{\beta \alpha} \tau^{-\beta p}(\tau \|u\|_{r})^{\beta p}d \tau \\
&\leqslant  \left(\sup_{t>0} t^{\beta} \|u\|_{r}\right) ^{p} \left(\sup_{t>0} t^{\beta} \|u\|_{q}\right)^{\alpha} \int_0^t (t-\tau)^{-\frac{N}{2r} (p-1)}\tau^{-\beta(p+\alpha)} d \tau \\
&\leqslant  \nu^{p+\alpha} t^{-\frac{N}{2r} (p-1)-\beta p- \beta \alpha+1} \int_0^1 (1-s)^{-\frac{N}{2r} (p-1)} s^{-\beta(p+\alpha)} ds \\
&= C \nu^{p+\alpha} t^{-\frac{N}{2q} (p-1) - \beta p - \beta \alpha +1} \mb\left( 1 - \beta (p+\alpha), 1 - \frac{N}{2r} (p-1) \right) \\
&= C \nu^{p+\alpha} t^{-\beta}, \quad t > 0,
\end{align}
We observe that, by \eqref{choice-r} and the condition \(\beta (p+\alpha) < 1\), the quantity 
$\mb\left( 1 - \beta (p+\alpha), 1 - \frac{N}{2r} (p-1) \right)$
is well-defined. Similarly, we can conclude that
\begin{align}
\int_0^t \| \tau^\varrho e^{(t-\tau) \Delta} \bw \|_{r} \, d\tau &\leq  \left(\int_0^t \tau^\varrho (t-\tau)^{-\frac{N}{2} \left( \frac{1}{l} - \frac{1}{r} \right)} d\tau\right) \|\bw \|_{\ell} \\
&=  t^{\varrho-\frac{N}{2} \left( \frac{1}{l} - \frac{1}{r} \right) + 1} \left(\int_0^1 s^\varrho (1-s)^{-\frac{N}{2}\left(\frac{1}{l}-\frac{1}{r}\right)} ds \right)\|\bw \|_{\ell}\\
&=  t^{\varrho-\frac{N}{2} \left( \frac{1}{l} - \frac{1}{r} \right) + 1}  \mb\left( 1+\varrho, 1 - \frac{N}{2} (\frac{1}{l}-\frac{1}{r}) \right)\|\bw \|_{\ell} \\
&= C t^{-\beta} \|\bw \|_{\ell}, \quad t > 0.
\end{align}
We emphasize that, under the assumption \(\varrho > -1\) and condition \eqref{choice-r}, the quantity 
$\mb \left( 1+\varrho, 1 - \frac{N}{2} \left( \frac{1}{l} - \frac{1}{r} \right) \right)
$ is well-defined. Consequently, the above estimates enable us to conclude that
\begin{equation}
\label{Stability1}
t^\beta \|\Phi(u)(t) \|_{r} \leq C \left( \| u_0 \|_{p_c} + \nu^{p+\alpha} + \|\bw\|_{\ell} \right), \quad t > 0
\end{equation}
By selecting \(\epsilon > 0\) and \(\nu > 0\) sufficiently small, we deduce that 
\begin{equation}
\label{Stability2}
\sup_{t>0} t^\beta \|\Phi(u)(t) \|_{r} \leq \nu.
\end{equation}
This proves that \(\Phi\) maps \(\mathbf{Y}\) into itself. To complete the proof, we now show that \(\Phi: \mathbf{Y} \rightarrow \mathbf{Y}\) is a contraction.

Using arguments similar to those above, we can show that for any \(u, v \in \mathbf{Y}\), the following holds:
\begin{align*}
\|\Phi(u)(t) - \Phi(v)(t)\|_{r} &\leq \int_0^t \left\| e^{(t-\tau)\Delta} \left( |u(\tau)|^p \|u(\tau)\|_{q}^\alpha - |v(\tau)|^p \|v(\tau)\|_{q}^\alpha \right) \right\|_{{r}} \, d\tau \\
&\leq \int_0^t (t-\tau)^{-\frac{N}{2}\left(\frac{p}{r} - \frac{1}{r}\right)} \left\| |u(\tau)|^p \|u-v\|_{q} \left( \|u(\tau)\|_{q}^{\alpha-1} - \|v(\tau)\|_{q}^{\alpha-1} \right) \right\|_{{\frac{r}{p}}} \, d\tau \\
&\quad +  \int_0^t (t-\tau)^{-\frac{N}{2}\left(\frac{p}{r} - \frac{1}{r}\right)} \|v(\tau)\|_{q}^\alpha \left\| |u-v|\left(|u(\tau)|^{p-1} - |v(\tau)|^{p-1}\right) \right\|_{{\frac{r}{p}}} \, d\tau \\
&\leq  \int_0^t (t-\tau)^{-\frac{N}{2}\left(\frac{p}{r} - \frac{1}{r}\right)} \tau^{-\beta p} (\tau^{\beta} \|u\|_{r})^{p} \tau^{-\beta} (\tau^{\beta} \|u-v\|_{r}) \tau^{-\beta(\alpha-1)} (\tau^{\beta} \|u\|_{r})^{\alpha-1} \, d\tau \\
&\quad +  \int_0^t (t-\tau)^{-\frac{N}{2}\left(\frac{p}{r} - \frac{1}{r}\right)} \tau^{-\beta \alpha} (\tau^{\beta} \|v(\tau)\|_{q})^\alpha \tau^{-\beta} (\tau^{\beta} \|u-v\|_{q}) \tau^{-\beta(p-1)} \\
&\qquad \times \left( (\tau^{\beta} \|u\|_{r})^{p-1} + (\tau^{\beta} \|v\|_{r})^{p-1} \right) \, d\tau \\
&\leq  t^{-\beta} \nu^{p+\alpha-1} \mb\left( 1 - \beta (p+\alpha), 1 - \frac{N}{2r} (p-1) \right) d(u,v).
\end{align*}
Therefore,
\[
d(\Phi(u), \Phi(v)) \leq C \nu^{p+\alpha-1} d(u, v).
\]
By choosing \(\nu > 0\) sufficiently small, we conclude that \(\Phi: \mathbf{Y} \rightarrow \mathbf{Y}\) is a contraction. Consequently, the global existence of a solution follows from the Banach fixed point theorem applied to the complete metric space \(\mathbf{Y}\).
\vspace{1cm}

\hrule 

\vspace{0.3cm}
\noindent{\bf\large Declarations.} {\em On behalf of all authors, the corresponding author states that there is no conflict of interest. No data-sets were generated or analyzed during the current study.}

	\vspace{0.7cm}

 \hrule


\begin{thebibliography}{99}

 \bibitem{Opuscula} {A. Alshehri, N. Aljaber, H. Altamimi, R. Alessa and M. Majdoub}, {\em Nonexistence of global solutions for a nonlinear parabolic             equation with a forcing term}, {Opuscula Math.}, {\bf 43} (2023), {741--758}.
\bibitem{AW78} D. G. Aronson and H. F. Weinberger, {\em Multidimensional nonlinear diffusion
arising in population genetics}, Adv. in Math., {\bf 30} (1978), 33--76.

\bibitem{Bandle89} C. Bandle, {\em Blow up in exterior domains}, in: Recent Advances in Nonlinear Elliptic and Parabolic Problems (Nancy, 1988), in: Pitman Res. Notes Math. Ser., vol. 208, Longman, Harlow, 1989, pp. 15--27.

\bibitem{BL89} C. Bandle and H. A. Levine {\em On the existence and nonexistence of global solutions of reaction-diffusion equations in sectorial domains}, Trans. Amer. Math. Soc., {\bf 655} (1989), 595--624.


\bibitem{BLZ} C. Bandle, H. A. Levine and  Qi S. Zhang, {\em Critical Exponents of Fujita Type for Inhomogeneous Parabolic Equations and Systems}, J. Math. Anal. Appl., {\bf 251} (2000), 624--648.

\bibitem{BC} {H. Brezis and T. Cazenave}, {\em A nonlinear heat equation with singular initial data}, {J. Anal. Math.},  {\bf 68} (1996), {277--304}.

\bibitem{CDW} {T. Cazenave, F. Dickstein and  F. B. Weissler}, {\em An equation whose {F}ujita critical exponent is not given by              scaling}, {Nonlinear Anal.}, {\bf 68} ({2008}), {862--874}.

\bibitem{DL} K. Deng and H. A. Levine, {\em The role of critical exponents in blow-up theorems: the sequel}, J. Math. Anal. Appl., {\bf 243} (2000), 85--126.


\bibitem{DM1986} J. Dixon and S. Mckee, {\em Weakly singular discrete Gronwall inequalities}, Z. Angew. Math. Mech., {\bf 66} (1986), 535--544.

\bibitem{Fried} A. Friedman, {\em Partial Differential Equations of Parabolic Type}, Northwestern University, Prentice-Hall, 1964.

\bibitem{fujita}{H. Fujita}, {\em On the blowing up of solutions of the Cauchy problem for $u_t=\Delta u+u^{1+\alpha}$}, J. Fac. Sci. Univ. Tokyo Sec. IA Math., {\bf 13} (1966), 109--124.
    
\bibitem{Fuj69} H. Fujita, {\em On some nonexistence and nonuniqueness theorems for nonlinear parabolic
equations}, Proc. Sympos. Pure Math., {\bf 18} (1969), 105--113.


\bibitem{Furter} {J. Furter
and M. Grinfeld}, {\em Local vs.\ nonlocal interactions in population dynamics}, {J. Math. Biol.}, {\bf 27} (1989), {65--80}.


\bibitem{Gal83} V.A. Galaktionov, {\em Conditions for nonexistence in the large and localization of solutions of the Cauchy problem for a class of nonlinear parabolic equations}, Zh. Vychisl. Mat. Mat. Fiz., {\bf 23} (1983), 1341--1354 (in Russian).

\bibitem{Gal80} {V.A. Galaktionov,  S. P. Kurdjumov, A. P. Miha\u ilov               and A. A. Samarski\u i}, {\em On unbounded solutions of the {C}auchy problem for the
              parabolic equation {$u\sb{t}=\nabla (u\sp{\sigma }\nabla
              u)+u\sp{\beta }$}}, {Dokl. Akad. Nauk SSSR}, {\bf 252} (1980), {1362--1364} (in Russian).
\bibitem{GL1998} V. A. Galaktionov and H. A. Levine, {\em A general approach to critical Fujita exponents in nonlinear parabolic problems}, Nonlinear Anal., {\bf 34} (1998), 1005--1027.

\bibitem{GV} { V. A. Galaktionov and J. L. V\'azquez}, {\em The problem of blow-up in nonlinear parabolic equations}, Discrete Contin. Dyn. Syst., \textbf{8} (2002), 399--433.


\bibitem{Gued-Kira} {M. Guedda and M. Kirane}, {\em Local and global nonexistence of solutions to semilinear evolution equations}, {Proceedings of the 2002 {F}ez {C}onference on {P}artial
              {D}ifferential {E}quations}, {Electron. J. Differ. Equ. Conf.},
{\bf 9} (2002), {149--160}.

\bibitem{Hayak} {K. Hayakawa}, {\em On nonexistence of global solutions of some semilinear parabolic differential equations}, Proc. Japan Acad., {\bf 49} (1973), 503--505.

\bibitem{Henry} {D. Henry, {\em Geometric Theory of Semilinear Parabolic Equations}, Lecture Notes in Mathematics, \textbf{840}, Springer-Verlag, Berlin-New York, 1981.}


\bibitem{Hu} {B. Hu},  {\em Blow Up Theories for Semilinear Parabolic Equations}, Springer, Berlin (2011).


\bibitem{JKS} {M. Jleli, T. Kawakami and B. Samet}, {\em Critical behavior for a semilinear parabolic equation with forcing term depending of time and space}, J. Math. Anal. Appl., {\bf 486} (2020), 123931.


\bibitem{KST77} K. Kobayashi, T. Siaro, and H. Tanaka, {\em On the blowing up problem for semilinear heat equations}, J. Math. Soc. Japan, {\bf 29} (1977), 407--424.

\bibitem{Lap2002} G.-I. Laptev, {\em Conditions for the nonexistence of global solutions of the {C}auchy problem for a parabolic equation with integral nonlinear perturbation}, {Differ. Uravn.}, {\bf 38} (2002), {547--554, 576}.
\bibitem{Lap2005} G.-I. Laptev, {\em Critical Exponents for the Heat-Conduction Equation with Nonlocal, Nonlinear Perturbations}, J. Math. Sci, {\bf 129} (2005), 3617--3625.


\bibitem{PGL} {P. G. Lemari{\'e}-Rieusset}, {\em Recent developments in the {Navier}-{Stokes} problem}, {Chapman Hall/CRC Res. Notes Math.}, Vol. {\bf 431}, {2002}.

\bibitem{LM89} {H. A. Levine and P. Meier, {\em The value of the critical exponent for reaction-diffusion equations in cones}, Arch. Rational Mech. Anal., \textbf{109} (1989), 73--80.}

\bibitem{Lev90} H.A. Levine, {\em The role of critical exponents in blow-up theorems}, SIAM Rev., {\bf 32} (1990), 262--288.

\bibitem{MM} M. Majdoub, {\em Well-posedness and blow-up for an inhomogeneous semilinear parabolic equation}, Differ. Equ. Appl., {\bf 13} (2021), {85--100}.

\bibitem{Majd} M. Majdoub, {\em On the Fujita Exponent for a Hardy-H\'enon Equation with a Spatial-Temporal Forcing Term}, La Matematica, {\bf 2} (2023), 340--361.


\bibitem{MP92} {R. H. Martin and M. Pierre, {\em Nonlinear reaction-diffusion systems}, {in Nonlinear Equations in the Applied Sciences}, Math. Sci. Engrg., 185 W. F. Ames and C. Rogers, Editors, Academic Press, Boston, MA 1992, 363-398.}


\bibitem{MT1} { A. V. Martynenko and A. F. Tedeev}, {\em Cauchy problem for a quasilinear parabolic equation with a source term and an inhomogeneous density}, Comput. Math. Math. Phys., \textbf{47} (2007), 238--248.

\bibitem{MT2} { A. V. Martynenko and A. F. Tedeev}, {\em On the behavior of solutions to the Cauchy problem for a degenerate parabolic equation with inhomogeneous density and a source}, Comput. Math. Math. Phys., \textbf{48} (2008), 1145--1160.



\bibitem{Mei86} {P. Meier}, {\em Existence et nonexistence de solutions globales d'une
              \'equation de la chaleur semi-lin\'eaire: extension d'un
              th\'eor\`eme de {F}ujita}, {C. R. Acad. Sci. Paris S\'er. I Math.}, {\bf 303} (1986), {635--637}.

\bibitem{Mei88} {P. Meier}, {\em Blow-up of solutions of semilinear parabolic differential
              equations}, {Z. Angew. Math. Phys.},{\bf 39} (1988), {135--149}.

\bibitem{Mei90} {P. Meier}, {\em On the critical exponent for reaction-diffusion equations}, {Arch. Rational Mech. Anal.}, {\bf 109} (1990), {63--71}.

              
\bibitem{ES} E. Mitidieri, S. I. Pokhozhaev, {\em A priori estimates and blow-up of solutions to nonlinear partial differential equations and inequalities}, Trudy Mat. Inst. Steklova, {\bf 234} (2001), 3--383.


\bibitem{Pao} {C. V. Pao}, {\em  Nonlinear parabolic and elliptic equations}, New York and London: Plenum Press; 1992.

\bibitem{Qi} { Y.-W Qi}, {\em  The critical exponents of parabolic equations and blow-up in $\R^n$}, Proc. Roy. Soc. Edinburgh Sect. A, \textbf{128} (1998), 123--136.

\bibitem{QS} {P. {Quittner} and P. {Souplet}}, {\em Superlinear parabolic problems. Blow-up, global existence and steady states}, 2nd revised and updated edition, {{Birkh\"auser Adv. Texts, Basler Lehrb\"uch.}}, {2019}.


\bibitem{Rothe} F. Rothe, {\em Global Solutions of Reaction-Diffusion Systems}, Lecture Notes in Mathematics, Vol. \textbf{1072}, Springer-Verlag, Berlin, 1984.


\bibitem{RT} {B. Ruf and E. Terraneo}, {\em The {Cauchy} problem for a semilinear heat equation with singular initial data}, {Evolution equations, semigroups and functional analysis. In memory of Brunello Terreni. Containing papers of the conference, Milano, Italy, September 27-28, 2000}, {295--309}, {2002}.


\bibitem{Salsa} {S. Salsa}, {\em Partial differential equations in action. From modelling to theory}, {{Unitext}}, Vol. \textbf{99}, {2016}.

\bibitem{SGKM} A. A. Samarskii, V. A. Galaktionov, S. P. Kurdyumov, and A. P. Mikhailov, {\em Blow-up
in Quasi-Linear Parabolic Equations}, Walter de Gruyter, Berlin (1995), page 535.

\bibitem{Soup98} P. Souplet, {\em Blow-up in nonlocal reaction–diffusion equations}, SIAM J. Math. Anal., {\bf 29} (1998) 1301--1334.

\bibitem{Soup99} P. Souplet, {\em Uniform blow-up profiles and boundary behaviour for diffusion equations with nonlocal nonlinear source}, J. Differential Equations, {\bf 153} (1999), 374--406.


\bibitem{Terr1} {E. Terraneo}, {\em On the non-uniqueness of weak solutions of the nonlinear heat equation with nonlinearity {{$u^3$}}}, {C. R. Acad. Sci., Paris, S{\'e}r. I, Math.}, {\bf 328} (1999), {759--762}.

\bibitem{Terr2}{E. Terraneo}, {\em Non-uniqueness for a critical nonlinear heat equation}, {Commun. Partial Differ. Equations}, {\bf 27} (2002), {185--218}.

 \end{thebibliography}
\end{document}